\documentclass{siamltex}

\usepackage{url}
\usepackage{algorithm,algorithmic}
\usepackage{amsfonts}
\usepackage{amssymb} % ou \usepackage{amsfonts}
\usepackage{booktabs}
\usepackage{amsmath}
\usepackage{graphicx}
\usepackage{xcolor}

\newcommand{\R}{\mathbb{R}}

\newtheorem{remark}{Remark}
\newcommand{\nd}{\noindent}

\title{Efficient Computation of Dominant Eigenvalues Using Adaptive Block Lanczos with Chebyshev Filtering}
\author{M. El Guide\thanks{AIRESS, FGSES-UM6P, Rabat, Morocco}
	\and K. Jbilou\thanks{Université du Littoral Cote d'Opale, LMPA, 50 rue F. Buisson, 62228 Calais-Cedex, France.} \and K. Lachhab\thanks{FST, Université Cadi Ayyad, Marrakech, Maroc.}}

\begin{document}
	\maketitle

%
% RECOMMENDED %%%%%%%%%%%%%%%%%%%%%%%%%%%%%%%%%%%%%%%%%%%%%%%%%%%
%
% to typeset URLs, URIs, and DOIs
%\usepackage{url}
%\usepackage{algorithm,algorithmic}
%\usepackage{amsfonts}
%\usepackage{amssymb} % ou \usepackage{amsfonts}
%\usepackage{booktabs}
%\usepackage{amsmath}
%\usepackage{graphicx}
%\usepackage{xcolor}
%\newtheorem{remarque}{Remark}
%\newcommand{\R}{\mathbb{R}}
%\newcommand{\C}{\mathbb{C}}
%\newcommand{\N}{\mathbb{N}}
%\newcommand{\nd}{\noindent}
%\newcommand{\fin}{\,\rule{1ex}{1ex}\,}

%\begin{document}
%\mainmatter              % start of a contribution
%%
%\title{Efficient Computation of Dominant Eigenvalues Using Adaptive Block Lanczos with Chebyshev Filtering}
%%
%\titlerunning{Block Lanczos with Chebyshev Filtering}  % abbreviated title (for running head)
%%                                     also used for the TOC unless
%%                                     \toctitle is used
%%
%\author{M. El Guide\inst{1} \and K. Jbilou\inst{2}\and 
%K. Lachhab\inst{3}}
%%
%\authorrunning{M. El Guide et al.} % abbreviated author list (for running head)
%%
%%%%% list of authors for the TOC (use if author list has to be modified)
%\tocauthor{M. El Guide, K. Jbilou and  K. Lachhab}
%%
%\institute{AIRESS, FGSES-UM6P, Rabat, Morocco.\\
%\email{mohamed.elguide@um6p.ma},
%\and
%Universit\'e du Littoral Côte d'Opale, Laboratoire  LMPA, Calais, France.\\
%\email{khalide.jbilou@univ-littoral.fr}
%\and
%FST, Université Cadi Ayyad, Marrakech, Maroc.\\
%\email{kamal.lachhab@ced.uca.ac.ma}
%}
%
%\maketitle              % typeset the title of the contribution

\begin{abstract}
We present an efficient method for computing dominant eigenvalues of large, nonsymmetric, diagonalizable matrices based on an adaptive block Lanczos algorithm combined with Chebyshev polynomial filtering. The proposed approach improves numerical stability through two key components: (i) the Adaptive Block Lanczos (ABLE) method, which maintains biorthogonality using SVD-based stabilization, and (ii) Chebyshev filtering, which enhances spectral separation via iterative polynomial filtering.  Numerical experiments on dense and sparse test problems confirm the effectiveness of the ABLE–Chebyshev algorithm, showing significantly improved accuracy and convergence compared to standard block Lanczos, especially in challenging spectral regimes with clustered or tightly packed eigenvalues.
% We would like to encourage you to list your keywords within
% the abstract section using the \keywords{...} command.
{\bf keywords}{block Lanczos method,	Krylov subspace methods, Chebyshev filtering, non-hermitian eigenvalue problems, Ritz values, polynomial filtering, biorthogonalization}
\end{abstract}
\section{Introduction}

In many areas of applied science, one frequently faces the challenge of computing eigenvalues of large, nonsymmetric square matrices. Eigenvalues with the largest real parts are often of particular interest, especially in stability analyses. As highlighted by Kerner \cite{kerner}, such computations are essential in various physical applications. This motivates the development of numerical techniques for solving the following eigenvalue problem:
\begin{equation}
	A X = X \Lambda, \label{equ1}
\end{equation}
where $A \in \mathbb{R}^{n \times n}$, $X \in \mathbb{R}^{n \times s}$, and $\Lambda = \mbox{diag}(\lambda_1,\ldots,\lambda_s)$, with $\lambda_i$ denoting eigenvalues of $A$ for $i=1,\ldots,s$. 

Many efficient numerical algorithms for large-scale eigenvalue computations are built upon the Lanczos procedure, which is a method for the successive reduction of a general matrix to tridiagonal form \cite{lanczos}. Significant advances in the 1970s and 1980s established its effectiveness for solving large, sparse Hermitian eigenvalue problems. More recently, attention has shifted toward adapting the Lanczos framework for non-Hermitian problems, particularly through block versions of the algorithm. For instance, Bai, Day, and Ye introduced the Adaptive Block Lanczos (ABLE) method for large-scale non-Hermitian eigenvalue problems \cite{bai}; see also Riquet’s thesis \cite{Riquet} for further developments. A.H. Bentbib, M. El Guide and K. Jbilou \cite{bentbib} proposed a version of the block Lanczos method for ill-posed problems. In \cite{barkouki,barkouki1} , the authors presented rational block Lanczos variants applied to linear dynamical systems in a model reduction context. In \cite{freund}, the author proposed several algorithmic variants, including the Hermitian and non-Hermitian band Lanczos methods, as well as a specialized Lanczos procedure for complex symmetric matrices. These techniques fall within the class of projection methods onto Krylov subspaces. 
Saad \cite{saad3} demonstrated that such Krylov-based methods can effectively approximate the dominant part of the spectrum. However, their convergence deteriorates when the eigenvalues of interest are not well-separated from the rest of the spectrum. In practice, achieving satisfactory convergence often demands a large Krylov subspace, which increases both computational and memory costs.
To address these limitations, various restart strategies have been developed in the literature, where the algorithm is periodically reinitialized using the best available approximations to the desired eigenvectors, thereby reducing the subspace dimension while preserving convergence \cite{saad_restart,arpack,stewart_krylovschur,sorensen_polyfilter,Kamal_bentbib}.

In this work, we propose a complementary strategy that combines adaptive block Krylov methods with spectral filtering. Specifically, we enhance the block Lanczos method with a Chebyshev polynomial filter that selectively amplifies components in the dominant eigenspace. This filtering is embedded within the adaptive ABLE framework to stabilize the biorthogonalization process and improve robustness. The resulting ABLE–Chebyshev algorithm achieves both efficient convergence and numerical stability, especially in cases where the spectral gap is small or the eigenvalues are tightly clustered.

The remainder of this paper is organized as follows. Section \ref{sec2} introduces the block Lanczos algorithm and outlines its key properties.
In Section \ref{sec3}, we examine how this method approximates eigenvalues through Krylov subspace projections. Section \ref{sec4} investigates the convergence behavior of the block Lanczos algorithm on diagonalizable matrices and establishes its relationship to an associated normal matrix. Section \ref{sec5} presents a Chebyshev variant of the block Lanczos method designed to enhance convergence. In Section \ref{sec6}, we introduce a more robust version of the algorithm: the Adaptive Block Lanczos (ABLE) method, and extend it with Chebyshev filtering for improved stability and spectral targeting.
Section \ref{sec7} reports numerical experiments that validate the performance of the proposed methods. Finally, Section \ref{sec8} concludes the paper with a summary of results.
\section{Block Lanczos Process}\label{sec2}
\subsection*{The standard method}	
In many situations it is desirable to work with a block of $s$ starting vectors rather than a single starting vector. This can be achieved using block generalizations of Krylov subspace methods, in which $A$ operates on a group of $s$ vectors at each step rather than on a single vector. 
\\ Several methods exist for constructing a biorthogonal system between the block Krylov subspaces ${\cal K}_k(A; V_1)$ and ${\cal K}_k(A^T; W_1)$. For instance, Jbilou and Sadok \cite{jbilou,jbilou1} described a global Lanczos process, and Freund \cite{freund} presented a band Lanczos method. In this paper, we focus on the block biorthogonal Lanczos process, which extends Saad’s classic Lanczos algorithm  \cite{saad} from $s=1$ (single-vector case) to $s>1$. Algorithm \ref{bla} gives the classical block Lanczos procedure for constructing biorthogonal bases for the Krylov subspaces.
%The classical block Lanczos method is described as follows: 
\begin{algorithm}
	\caption{Block Lanczos method}\label{bla}
	\begin{enumerate}
		\item Choose initial $V_1, W_1 \in \R^{n \times s}$ such that $W_1^T V_1 = I_{s}$.
		\item Set $B_1 = C_1 = 0_{s \times s}$ and $W_0 = V_0 = 0_{n \times s}$.
		\item For $j = 1, 2, \ldots, k$ do:
		\begin{enumerate}
			\item $A_j = W_j^T A V_j$.
			\item $\widehat{V}_{j+1} = A V_j - V_j A_j - V_{j-1} B_j$.
			\item $\widehat{W}_{j+1} = A^T W_j - W_j A_j^T - W_{j-1} C_j^T$.
			\item Compute the QR decompositions: ${\widehat W}_{j+1}^T {\widehat V}_{j+1} = B_{j+1} C_{j+1}$.
			\item $W_{j+1} = \widehat{W}_{j+1} B_{j+1}$; \quad $V_{j+1} = \widehat{V}_{j+1} B_{j+1}^{-1}$.
		\end{enumerate}
	\end{enumerate}
\end{algorithm}
The block Lanczos method builds a biorthogonal system of basis vectors $\{V_j\}_{j=1}^k$ and $\{W_j\}_{j=1}^k$ for the two Krylov subspaces ${\cal K}_k(A; V_1)$ and ${\cal K}_k(A^T; W_1)$, respectively, satisfying:
$$
\left \{
\begin{array}{ccc}
	W_i^TV_i & =I_{s \times s} & i=1,\ldots,k, \\
	~~~ & ~~~~ & ~~~\\
	W_i^TV_j & =0_{s \times s} & j,i=1,\ldots,k;~i \neq j.
\end{array}
\right.
$$
\noindent
In this algorithm, a breakdown may occur if any $C_{j+1}$ is singular (i.e., if $\widehat{W}_{j+1}$ and $\widehat{V}_{j+1}$ fail to have full column rank). 
The main drawbacks of the nonsymmetric block method is that  loss of biorthogonality may occur, requiring reorthogonalization or restarting, and there is also a risk of breakdown if a block becomes rank-deficient. At the end of the paper, we will give a more stable version of the algorithm. 
%\end{proposition}

% \noindent{\bf Proof.}~~\\
% 1). For the first part of the proposition, consider Step $3(e)$ of the block Lanczos algorithm~1, $\forall i=1,\ldots,k$, we have
% $$W_i^TV_i=b_i^T \hat{W}_{i}^T \hat{V}_{i}c_i^{-1}.$$ 
% Moreover, for all $$i=1,\ldots,k,~~ \hat{W}_{i}^T \hat{V}_{i}=b_ic_i,$$
% where $b_i$ is a orthogonal matrix. Therefore $$W_i^TV_i=I_{s \times s},
% ~~ for~~ i=1,\ldots,k$$.
% \vskip 4mm
% 2). We next show that $W_j^TV_i=0_{s \times s}$ si $j \neq i$ where $W_j$ and $V_i$ are obtained from the block Lanczos algorithm. The proof proceeds naturally by induction.
% \vskip 4mm
% 3). To prove that  $W_i^TV_{j+1}=0_{s \times s}$, $\forall i \leq j$,
% we will precced the same technique of induction, as previously, this time by replacing
% $V_{j+1}$.
% \vskip 4mm
% 4). The relations in (\ref{equ2}) are simply a matrix reformulation of the recurrence equations
% $$AV_j=V_ja_j+V_{j-1}b_j+V_{j+1}c_{j+1}
% \mbox{~~and~~} A^TW_j=W_ja_j^T+W_{j-1}c_j^T+W_{j+1}b_{j+1}^T,$$ $ \forall j=1, \ldots ,k$.
% \fin \\
% ~~~~\\
% The relations $A {\cal V}_k = {\cal V}_{k+1}T_k$, $A^T {\cal W}_k = {\cal W}_{k+1}T_k^T$
% and ${\cal W}_k^TA {\cal V}_k = \tilde{T}_k$ this will allow us to compute approximate eigenvalues of $A$ (Ritz values) from the smaller matrices $\tilde{T}_j$,  as detailed in the following section.

\nd If the block Lanczos  algorithm runs for $k$ iterations without abreakdown, then the block vectors $V_1,\ldots,V_k$ and $W_1,\ldots,W_k$ form a biorthogonal system, i.e.,
$$
\left \{
\begin{array}{ccc}
	W_i^TV_i & =I_{s \times s} & i=1,\ldots,k ,\\
	~~~ & ~~~~ & ~~~\\
	W_i^TV_j & =0_{s \times s} & j,i=1,\ldots,k;~i \neq j.
\end{array}
\right.
$$
\nd Setting $\mathcal{V}_{k}=[V_{1},V_{2},...,V_{k} ]$ and $\mathcal{W}_{k}=[ W_{1},W_{2},...,W_{k} ]$, we have the following block Lanczos relations
\begin{equation}\label{eqLanc1}
	\begin{aligned}
		A\mathcal{V}_{k} &= \mathcal{V}_{k}T_{k}+V_{k+1}B_{k+1}E_{k}^{T},\\
		&= \mathcal{V}_{k+1}{\widetilde T} _{k},
	\end{aligned}
\end{equation}
and
\begin{equation}\label{eqLanc2}
	\begin{aligned}
		A^{T}\mathcal{W}_{k} &= \mathcal{W}_{k}T_{k}^{T}+W_{k+1}C_{k+1}^{T}E_{k}^{T},\\
		& = {\mathcal W}_{k+1} {\widetilde T}_k^T,
	\end{aligned}
\end{equation}
where $E_k$ is last $ks\times s$ block of the identity matrix $I_{ks}$ and ${\widetilde T}_k$ is the block tridiagonal matrix defined by
\[
{\widetilde T}_k=\left( 
\begin{array}{cccccc}
	A_{1} & B_{2} &  &  &  & \\ 
	C_{2} & A_{2} &  &  &  &  \\ 
	& \ddots  & \ddots   &  \ddots &  &  \\ 
	&  &  &  &  & B_{k} \\ 
	&  &  &  & C_{k} & A_{k}\\
	& & & & & C_{k+1}
\end{array}%
\right).
\]
The block tridiagonal matrix $T_k$ is of dimension $ks\times ks$ and obtained from  ${\widetilde T}_k$ by removing the last $s$ rows. 

%	\section{Connection between eigenvalues of $A$ and eigenvalues of $\tilde{T}_j$}
\section{Approximating A’s Eigenvalues from the Lanczos Projection}\label{sec3}

Computing a subset of eigenvalues of a large matrix $A$ provides important information about the matrix's spectral properties. Block Krylov processes such as Arnoldi and Lanczos allow us to estimate several of the eigenvalues of $A$. In order to do so, we first establish a link between the eigenvalues of $A$ and the eigenvalues of the smaller matrices $\tilde{T}_j$ for $j=1,\ldots,k$.
For each step $j$, partition $T_{j+1} = \begin{pmatrix}T_j \ C_{j+1}^T\end{pmatrix}$ where $T_j\in\R^{js\times js}$ and $C_{j+1}$ is the last block row. Then $A {\cal V}_j = {\cal V}_{j+1} \begin{pmatrix}T_j \ C_{j+1}^T\end{pmatrix}$, which we can write as $A {\cal V}_j = {\cal V}_j \widetilde{T}_j + V_{j+1}C_{j+1}$, i.e.,
%From the relation $A\, {\cal V}_j = {\cal V}_{j+1} T_j$ (for $j=1,\ldots,k$), we can isolate the part involving $\tilde{T}_j$ as follows:
\begin{equation}
	A {\cal V}_j ={\cal V}_j \tilde{T}_j+[0_{n \times s}; \ldots ;0_{n \times s};V_{j+1}C_{j+1}];
	\mbox{~~~} \forall j=1, \ldots ,k, \label{equ3}
\end{equation}
where $[0_{n \times s}; \ldots ;0_{n \times s};V_{j+1}C_{j+1}] \in \R^{n \times (j+1)s}$. The last term collects the contributions in the direction of $V_{j+1}$.\\
~~~\\
Let $z_i \in \R^{js}$ be an eigenvector of $\tilde{T}_j$ associated with eigenvalue $\lambda_i$ (for $i=1,\ldots, js$, counting multiplicities). And define $Z_{j,s}\in\R^{js\times js}$ to be the matrix whose columns are these eigenvectors, $Z_{j,s} = [\, z_1 ;\; z_2 ;\; \cdots ;\; z_{js}\,] $. ~~~~\\
Multiplying (\ref{equ3}) by $Z_{j,s}$ yields:
$$
A {\cal V}_jZ_{j,s} ={\cal V}_j \tilde{T}_jZ_{j,s}
+[0_{n \times s}; \ldots ;0_{n \times s};V_{j+1}C_{j+1}]Z_{j,s}
;\mbox{~~~} \forall j=1, \ldots ,k.
$$
Moreover, since $\tilde{T}_j z_i = \lambda_i z_i$ for all $i=1,\ldots,js$ and $j=1,\ldots,k$, then  we obtain 
$$
\tilde{T}_jZ_{j,s}=Z_{j,s} \Lambda \mbox{;~~} \forall j=1, \ldots ,k \mbox{;~~where~~}
\Lambda =diag( \lambda _1, \ldots , \lambda _{js}).
$$
Substituting this expression into the previous equation yields:
$$
A {\cal V}_jZ_{j,s} ={\cal V}_j Z_{j,s} \Lambda
+[0_{n \times s}; \ldots ;0_{n \times s};V_{j+1}C_{j+1}]Z_{j,s}
\mbox{;~~} \forall j=1, \ldots ,k .
$$
Now define $Z_j = {\cal V}_j Z_{j,s} \in \R^{n \times js}$, for all $ j=1, \ldots ,k$, This matrix can be viewed as containing approximate eigenvectors of $A$ corresponding to $\lambda_1,\ldots,\lambda_{js}$. Let $\widetilde{Z}{j,s}$ denote the last $s$ rows of $Z_{j,s}$. We can partition $Z_{j,s}$ conformally as
$$
Z_{j,s}= \left (
\begin{array}{c}
	*\\
	\vdots\\
	*\\
	\widetilde{Z}_{j,s}
\end{array}
\right ).
$$
Then from the previous relation we obtain:
\begin{equation}\label{equa3_1}
	\|A Z_j - Z_j \Lambda \|_F = \|\, V_{j+1} c_{j+1}\, \tilde{Z}_{j,s} \,\|_F, \qquad \forall j = 1,\ldots,k.
\end{equation}
We use the notation $	\|.\|_F$ for the Frobenius norm, and we will continue to use it throughout the paper.\\
Equation \ref{equa3_1} shows that the Frobenius-norm error in the eigenpair approximations (for eigenvalues $\lambda_1,\ldots,\lambda_{js}$) is proportional to the norm of  $V_{j+1}C_{j+1}\widetilde{Z}_{j,s}$.\\

\section{Some properties of the block Lanczos process}\label{sec4}

%Among all matrices, certain classes have special properties (e.g., symmetric, Hermitian, or normal matrices). In this section, we show that the convergence behavior of the block Lanczos eigenvalue approximations for any diagonalizable matrix $A$ can be matched by its convergence on a certain related normal matrix. We further infer that the convergence behavior of the approximate eigenvectors obtained by block Lanczos on a diagonalizable matrix is no worse than its behavior on the corresponding normal matrix.
%%
Among all matrices, certain classes exhibit special structural properties such as symmetry, Hermitianity, or normality, which often lead to more favorable numerical behavior. In this section, we aim to show that the convergence behavior of eigenvalue approximations produced by the block Lanczos method, when applied to a general diagonalizable matrix $A$ can be matched by its convergence on a certain related normal matrix.\\
More precisely, we will demonstrate that, this implies one can analyze the (typically simpler) normal matrix case to understand the convergence bounds for the general case. This observation implies that analyzing convergence on normal matrices can serve as a useful proxy for understanding the behavior of the algorithm on more general cases.
\begin{remark}
	In the analysis below we focus on right block vectors, but analogous results hold for left block vectors.
\end{remark}

\begin{theorem}
	Let \( U \in \mathbb{R}^{n \times n} \) be an orthonormal matrix, and let \( A \in \mathbb{R}^{n \times n} \) be a real, nonsymmetric matrix. Define
	\begin{equation}
		\hat{A} = U^T A U.\label{equ15}
	\end{equation}
	If the block Lanczos algorithm is applied to both $A$ and $\hat{A}$ with suitably related starting blocks, then for every iteration $j$, producing the projected matrices \( \tilde{T}_{jA} \), \( \tilde{T}_{j\hat{A}} \), and the corresponding Krylov bases \( \mathcal{V}_{jL,A} \), \( \mathcal{W}_{jL,A} \), \( \mathcal{V}_{jL,\hat{A}} \), and \( \mathcal{W}_{jL,\hat{A}} \), for each iteration index \( j \geq 1 \). Assume that the starting blocks  satisfy \( V_{1L,\hat{A}} = U^T V_{1L,A} \) and \( W_{1L,\hat{A}} = U^T W_{1L,A} \).
	Then, the following relations hold for all \( j \):
	\[
	\tilde{T}_{jA} = \tilde{T}_{j\hat{A}}, \quad \mathcal{V}_{jL,A} = U^T \mathcal{V}_{jL,\hat{A}}, \quad {\rm and} \quad  \mathcal{W}_{jL,A} = U^T \mathcal{W}_{jL,\hat{A}}.
	\]
	As a consequence, both the Ritz values and the Frobenius norms of the residual and error  for the block Lanczos approximations applied to \( A \) and \( \hat{A} \) are identical.
\end{theorem}
%	\begin{theorem}
	%		Let $U$ be an $n \times n$ orthonormal matrix. Let $A$ be a real diagonalizable $n \times n$ matrix, and define 
	%		\begin{equation}
		%			\hat{A} = U^T A\, U. \label{equ15}
		%		\end{equation}
	%		For each $j \ge 1$, let $\tilde{T}_{j,A}$ and $\tilde{T}_{j,\hat{A}}$ denote the block Lanczos tridiagonal matrices obtained by applying Algorithm~1 to $A$ and to $\hat{A}$, respectively, using initial blocks $V_{1L,A}, W_{1L,A}$ for $A$, and $V_{1L,\hat{A}} = U^T V_{1L,A}$, $W_{1L,\hat{A}} = U^T W_{1L,A}$ for $\hat{A}$. Let ${\cal V}_{jL,A}, {\cal W}_{jL,A}$ and ${\cal V}_{jL,\hat{A}}, {\cal W}_{jL,\hat{A}}$ denote the corresponding Lanczos basis blocks for $A$ and $\hat{A}$. Then for each $j$, we have $\tilde{T}_{j,A} = \tilde{T}_{j,\hat{A}}$, ${\cal V}_{jL,A} = U\, {\cal V}_{jL,\hat{A}}$, and ${\cal W}_{jL,A} = U\, {\cal W}_{jL,\hat{A}}$. In particular, the block Lanczos algorithms on $A$ and $\hat{A}$ produce identical eigenvalue approximations at each step, and the residual and error Frobenius norms of the right block Ritz vectors are also identical at each step.
	%	\end{theorem}

%\noindent{\bf Proof.}
\begin{proof}
	From the block Lanczos relation (\ref{equ3}) we have, for each $j=1,\ldots,k$:
	$$
	\left \{
	\begin{array}{ccc}
		A {\cal V}_{jL,A} & = & {\cal V}_{jL,A} \tilde{T}_{jA}+[0_{n \times s}; \ldots ;
		0_{n \times s};V_{(j+1)L,A}C_{(j+1)A}],\\
		A^T {\cal W}_{jL,A} & = & {\cal W}_{jL,A} \tilde{T}^T_{jA}+[0_{n \times s}; \ldots ;
		0_{n \times s};W_{(j+1)L,A}B^T_{(j+1)L,A}].\\
	\end{array}
	\right.
	$$
	Now, using the decomposition $\hat{A} = U^T A U$, we obtain:
	$$
	\left \{
	\begin{array}{ccc}
		U \hat{A}U^T {\cal V}_{jL,A} & = & {\cal V}_{jL,A} \tilde{T}_{jA}+[0_{n \times s}; \ldots ;
		0_{n \times s};V_{(j+1)L,A}C_{(j+1)A}],\\
		U \hat{A}^TU^T  {\cal W}_{jL,A} & = & {\cal W}_{jL,A} \tilde{T}^T_{jA}+[0_{n \times s}; \ldots ;
		0_{n \times s};W_{(j+1)L,A}B^T_{(j+1)A}].\\
	\end{array}
	\right.
	$$
	Now multiply the first relation on the left by $U^T$ and using $U^TU=I_{n \times n}$, we obtain:
	$$
	\left \{
	\begin{array}{ccc}
		\hat{A}(U^T {\cal V}_{jL,A}) & = &U^T {\cal V}_{jL,A} \tilde{T}_{jA}+[0_{n \times s}; \ldots ;
		0_{n \times s};U^TV_{(j+1)L,A}C_{(j+1)A}],\\
		\hat{A}^T(U^T {\cal W}_{jL,A}) & = & U^T {\cal W}_{jL,A} \tilde{T}^T_{jA}+[0_{n \times s}; \ldots ;
		0_{n \times s};U^TW_{(j+1)L,A}B^T_{(j+1)A}].\\
	\end{array}
	\right.
	$$
	Applying relation $(\ref{equ3}) $ to the matrix $	\hat{A}$
	, and identifying corresponding terms, we obtain the following equalities:
	$$
	\tilde{T}_{jA}= \tilde{T}_{j \hat{A}} \mbox{,~~} {\cal V}_{jL,A}=U
	{\cal V}_{jL, \hat{A}} \mbox{~and~~} {\cal W}_{jL,A}= U{\cal W}_{jL, \hat{A}}.
	$$
	In the second part of the theorem, we make use of the fact that $A$ and
	$\hat{A}$ have the same eigenvalues.\\
	Let $\Lambda =diag( \lambda_1; \ldots ; \lambda_s)$, the $s$ eigenvalues of
	the largest real part
	of $A$, and $X=[x_1; \ldots ,x_s]$
	be their associated respective eigenvectors.
	Thus, using the expression in equation~\ref{equ15}, we observe that $\Lambda =diag( \lambda_1; \ldots ; \lambda_s)$ represents the 
	$s$ eigenvalues of $\hat{A}$ with the largest real parts. Moreover, the corresponding eigenvectors are given by $U^TX=[\hat{x}_1; \ldots ,\hat{x}_s]$, where each ${x}_i$ is associated with ${\lambda}_i$.
	\\
	Suppose that $$\tilde{T}_{jA}Y_{j \tilde{T}}=Y_{j \tilde{T}} \theta,$$ where
	$\theta =diag( \mu_1; \ldots ; \mu_s)$ the $s$ eigenvalues of the largest real part
	of $\tilde{T}_{jA}= \tilde{T}_{j \hat{A}}$, and
	$Y_{j \tilde{T}}=[y_{j \tilde{T}}; \ldots ;y_{j \tilde{T}}]$
	their associated eigenvectors.\\
	We define the Ritz eigenvector approximations of $A$ and $\hat{A}$ $Z_{A}= {\cal V}_{jL,A}Y_{j \tilde{T}}$ and $Z_{ \hat{A}}
	= {\cal V}_{jL, \hat{A}}Y_{j \tilde{T}}$, respectively. Given the relation  ${\cal V}_{jL,A}=U {\cal V}_{jL, \hat{A}}$ it then follows that $Z_{ \hat{A}}=U^TZ_{A}$.\\
	Therefore, we obtain:  
	$$
	\|E_{A} \|_F= \| Z_{A}-X \|_F.
	$$
	Given the orthonormality of $U^T$, we deduce that 
	\begin{eqnarray*}
		\|E_{A} \|_F & = & \|U^TZ_{A}-U^TX \|_F,\\
		& = & \| Z_{ \hat{A}}-U^TX\|_F, \\
		& = & \|E_{ \hat{A}} \|_F.
	\end{eqnarray*}
	Furthermore, we have the same residual (by using relation (\ref{equ15})):
	\begin{eqnarray*}
		\|R_{A} \|_F & = & \|AZ_{A}-Z_{A} \theta \|_F,\\
		& = & \|U \hat{A}U^TZ_{A}-Z_{A} \theta \|_F,\\
		& = & \|U (\hat{A}U^TZ_{A,L}-U^TZ_{A} \theta) \|_F.
	\end{eqnarray*}
	Applying the same properties of the Frobenius norm as before, we obtain:
	$$
	\|R_{A} \|_F=\| \hat{A}U^TZ_{A}-U^TZ_{A} \theta \|_F.
	$$
	Since, $Z_{ \hat{A}}=U^TZ_{A}$, we finally obtain
	$$
	\|R_{A} \|_F=\| \hat{A}Z_{ \hat{A}}-Z_{ \hat{A}} \theta \|_F=
	\|R_{ \hat{A}} \|_F \mbox{~~~} .
	$$
\end{proof}

\begin{remark}
	For the left block Ritz vectors, the corresponding transformations are applied using the transpose matrix $U^T$.
\end{remark}
Diagonalizable matrices admit multiple representations, particularly through structured decompositions involving orthonormal and diagonal components. In the context of this section, we utilize both the diagonalizability of $A$ and its decomposition involving orthonormal matrices. When combined with the result of the previous theorem, this construction allows us to establish that the convergence behavior of the block Lanczos process on $A$ is preserved under transformation -thus bringing us closer to the main objective of this section.\\
\begin{theorem}
	Let \( A \in \mathbb{R}^{n \times n} \) be a diagonalizable matrix with Jordan decomposition \( A = Z \chi Z^{-1} \). Let \( Z = U \Sigma P^T \) be the singular value decomposition of \( Z \), where \( U, P \in \mathbb{R}^{n \times n} \) are orthonormal, and \( \Sigma \) is a positive diagonal matrix. Define the normal matrix \( \hat{A}_N = P^T \chi P \), and set
	\[
	\hat{A} = \Sigma \hat{A}_N \Sigma^{-1}.
	\]
	Then, \( A \) is orthogonally similar to \( \hat{A} \), and the block Lanczos method applied to \( A \) and \( \hat{A} \) yields identical Ritz values and residual norms.
\end{theorem}

%\noindent{\bf Proof.}
\begin{proof}
	If we apply the Jordan decomposition to the matrix $A$, we obtain:
	$$
	A=Z \chi Z^{-1}.
	$$
	As a first step, we express $Z$ via its singular value decomposition:
	$$
	Z=U \Sigma P^T \mbox{~~with~~}U^TU=I_{n \times n} \mbox{~~and~~}P^TP=I_{n \times n}.
	$$
	From the decompositions, we have:
	\[
	A = U \Sigma P^T \chi P \Sigma^{-1} U^T = U \hat{A} U^T,
	\]
	Since $\hat{A}_N=P^T \chi P$, and $\chi$
	is diagonal, it follows that $\hat{A}_N$ is a normal matrix. Moreover, if we define $\hat{A}= \Sigma \hat{A}_N \Sigma^{-1}$,
	, then $\hat{A}$ is similar to $\hat{A}_N$ via a similarity transformation involving a positive diagonal matrix $\Sigma$.\\
	This demonstrates that \( A \) is orthogonally similar to \( \hat{A} \). By using the result of previous theorem, we conclude that applying the block Lanczos algorithm to both  \( A \) and \( \hat{A} \) yields the same convergence behavior, eigenvalue approximations, and residuals.\\
\end{proof}
Under the assumptions of two preceding theorems, we can therefore conclude that the application of block Lanczos algorithm to the matrices $\hat{A}_N $ and $\hat{A}$
implies identical convergence behavior for the eigenvalue approximations.
\begin{theorem}
	Let \( \hat{A} = \Sigma \hat{A}_N \Sigma^{-1} \), where \( \hat{A}_N \) is normal. Apply the block Lanczos algorithm to \( \hat{A} \) with starting vectors \( (V_{1L}, W_{1L}) \), and to \( \hat{A}_N \) with scaled starting vectors \( (\Sigma^{-1} V_{1L}, \Sigma W_{1L}) \). Then, for all \( j \), the resulting projected matrices and vectors satisfy:
	\[
	\tilde{T}_{j\hat{A}} = \tilde{T}_{j\hat{A}_N}, \quad \mathcal{V}_{jL,\hat{A}} = \Sigma \mathcal{V}_{jL,\hat{A}_N}, \quad \mathcal{W}_{jL,\hat{A}} = \Sigma^{-1} \mathcal{W}_{jL,\hat{A}_N}.
	\]
	Moreover, if \( Z_{\hat{A}} = \mathcal{V}_{jL,\hat{A}} Y_j \) and \( Z_{\hat{A}_N} = \mathcal{V}_{jL,\hat{A}_N} Y_j \), then
	\[
	Z_{\hat{A}_N} = \Sigma^{-1} Z_{\hat{A}}, \quad \mbox{~~and~~} \quad \|R_{\hat{A}}\|_F = \|\hat{A} Z_{\hat{A}} - Z_{\hat{A}} \Lambda\|_F = \|\Sigma R_{\hat{A}_N}\|_F.
	\]
\end{theorem}
%	\noindent{\bf Proof.}
\begin{proof}
	\vskip 2mm
	The first part of the proof follows directly from the reasoning presented in Theorem $2$.\\
	for the computation of $\| R_{ \hat{A}} \|_F$, we have:
	\begin{eqnarray*}
		\| R_{ \hat{A}} \|_F & = & \| \hat{A} Z_{ \hat{A}}-Z_{ \hat{A}} \theta \|_F, \\
		& = & \| \hat{A} \Sigma \Sigma^{-1}Z_{ \hat{A}}-\Sigma \Sigma^{-1}Z_{ \hat{A}} \theta \|_F.
	\end{eqnarray*}
	Yet, $Z_{ \hat{A}_N,L}= \Sigma^{-1} Z_{ \hat{A},L}$. Then, we obtain
	\begin{eqnarray*}
		\| R_{ \hat{A}} \|_F & = & \| \hat{A} \Sigma Z_{ \hat{A}_N}-\Sigma
		Z_{ \hat{A}_N} \theta \|_F,\\
		& = & \| \Sigma (\Sigma^{-1} \hat{A} \Sigma Z_{ \hat{A}_N}-Z_{ \hat{A}_N}
		\theta) \|_F.
	\end{eqnarray*}
	We know that $\hat{A}_N= \Sigma^{-1} \hat{A} \Sigma$, then
	$$
	\| R_{ \hat{A}} \|_F= \| \Sigma ( \hat{A}_N Z_{ \hat{A}_N}-Z_{ \hat{A}_N}
	\theta) \|_F,
	$$
	and, since we have $ R_{ \hat{A}_N}= \hat{A}_N Z_{ \hat{A}_N}-Z_{ \hat{A}_N}
	\theta$, we can concluse that:
	$$
	\| R_{ \hat{A}} \|_F=\| \Sigma R_{ \hat{A}_N} \|_F. \mbox{~~} 
	$$
This has the same form as $(\ref{equ3})$ (with $\tilde{T}_{j,A}$ in place of $\tilde{T}_{j,\hat{A}}$). By uniqueness of the recurrence, we must have $\tilde{T}_{j,\hat{A}} = \tilde{T}_{j,A}$ for each $j$. A similar argument can be applied to the second relation to show $\tilde{T}_{j,\hat{A}} = \tilde{T}_{j,A}$, consistent with the first result. It follows that the block Lanczos processes on $A$ and $\hat{A}$ produce the same projected matrices $\tilde{T}_j$ at each step. Moreover, $U\, {\cal V}_{jL,\hat{A}} = {\cal V}_{jL,A}$ and $U\, {\cal W}_{jL,\hat{A}} = {\cal W}_{jL,A}$, implying that the Ritz vectors for $\hat{A}$ are just $U^T$ applied to the Ritz vectors for $A$. Therefore, the residual norms and error norms for the right block Ritz vectors are identical for $A$ and $\hat{A}$ at every step. 
\end{proof}
Combining Theorems~1 and~2, we conclude that the convergence behavior of the block Lanczos method on any diagonalizable matrix $A$ is the same as its behavior on a certain normal matrix $\hat{A}_N$ with the same eigenvalues as $A$. In particular, the eigenvalue approximations and the residual and error norms for the Ritz vectors are identical for $A$ and $\hat{A}_N$.

\section{Block Lanczos--Chebyshev method for eigenvalue computation}\label{sec5}

The convergence of the eigenvalue approximations obtained by the block Lanczos process can sometimes be very slow, especially when the eigenvalues of interest are clustered or have small spectral gap. To overcome this difficulty, one approach is to use the method iteratively: the process is restarted periodically with the best available approximations to the desired eigenvectors. These approximate eigenvectors can be further refined via a Chebyshev iteration. The goal of the Chebyshev iteration is to amplify the components of the restart vectors in the desired eigendirections while attenuating those components corresponding to unwanted eigenvalues. Manteuffel \cite{mant1,mant2} used Chebyshev iterative filtering to accelerate the solution of linear systems. %Sadkane's %Arnoldi-Chebyshev method \cite{sadkane2} is closely related to the procedure we examine in this section.	
Throughout this section, assume $A\in\R^{n\times n}$ is real, nonsymmetric (but diagonalizable), with eigenvalues $\lambda_1,\ldots,\lambda_n$ sorted in decreasing order of $\Re(\lambda)$. We aim to solve the eigenvalue problem
$$
A u = \lambda u.
$$
%where $A$ is a real, diagonalizable, nonsymmetric matrix of order $n$. Let the eigenvalues of $A$ be $\lambda_1, \ldots, \lambda_n$, labeled in decreasing order of their real parts.
Suppose we are interested in computing the leading $s$ eigenvalues $\lambda_1, \ldots, \lambda_s$ (with $s \le n$). The Chebyshev iteration approach is based on finding a suitably tight ellipse in the complex plane that contains all the unwanted eigenvalues $\{\lambda_{s+1}, \ldots, \lambda_n\}$.
We then seek a polynomial $P_k$ of degree $k$ hat acts as a spectral filter: the sequence defined by  $Z_k = P_k(A)\, Z_0$, where $Z_0 \in \R^{n \times s}$ an initial approximation to the target invariant subspace, is expected to converge to a matrix whose columns span a subspace that closely approximates the eigenspace associated with the eigenvectors $u_1,\ldots,u_s$. Let $T_k$ denote the Chebyshev polynomial of degree $k$ of the first kind \cite{saad2}. We choose 
$$
P_k(\lambda) = \frac{T_k\!\Big(\frac{\lambda - d}{c}\Big)}{T_k\!\Big(\frac{\lambda_s - d}{c}\Big)},
$$
where $d$ and $c$ are focal points, that $d$ is the center (real axis intersection) and $c$ the semi-major/minor distances. \\
By construction, $P_k(\lambda_s)=1$ and $|P_k(\lambda)| < 1$ for $\lambda$ in the unwanted set and  $|P_k(\lambda)| < 1$ for $\lambda$ outside the ellipse (specifically for $\lambda \in \{\lambda_{s+1},\ldots,\lambda_n\}$).
\\The associated convergence factor is defined by:
$$
R_i(d,c^2) = \frac{\Big|\, (d-\lambda_i) + \sqrt{(d-\lambda_i)^2 - c^2}\,\Big|}{\Big|\, (d-\lambda_s) + \sqrt{(d-\lambda_s)^2 - c^2}\,\Big| },
$$
for each unwanted eigenvalue $\lambda_i$. The choice of $d$ and $c$ that yields the smallest convergence factor (while enclosing all $\lambda_{s+1}, \ldots, \lambda_n$) defines the optimal Chebyshev ellipse. Ho \cite{ho2} developed a procedure to find this optimal ellipse containing the unwanted eigenvalues without having to test all possible ellipses.
\\
Recall that Chebyshev polynomials satisfy the recurrence:
$$
T_{i+1}(\lambda) = \lambda\, T_i(\lambda) - T_{i-1}(\lambda), \qquad i \ge 1,
$$
with $T_0(\lambda) = 1$ and $T_1(\lambda) = \lambda$.
\\
Assume that the eigenvalues $\mu_1, \ldots, \mu_{j s}$ of $\tilde{T}_j$ (for any $j \ge 2$) are labeled in decreasing order of their real parts. Then the sequence of matrices $Z_i$ ($i \ge 0$) can be computed using the following Chebyshev iteration procedure
\begin{algorithm}[h]
	\caption{Chebyshev iteration}\label{filtre}
	\begin{enumerate}
		\item $\displaystyle \sigma_1 = \frac{c}{\,\lambda_s - d\,}$; \quad $\displaystyle Z_1 = \frac{\sigma_1}{c}\, (A - d I)\, Z_0$.
		\item For $i = 1, 2, \ldots, k-1$ do:
		\begin{enumerate}
			\item $\displaystyle \sigma_{i+1} = \frac{1}{\,\frac{2}{\sigma_1} - \sigma_i\,}$.
			\item $\displaystyle Z_{i+1} = 2\, \frac{\sigma_{i+1}}{c}\, (A - d I)\, Z_i \;-\; \sigma_i\, \sigma_{i+1}\, Z_{i-1}$.
		\end{enumerate}
		\item End for.
	\end{enumerate}
\end{algorithm}

Algorithm \ref{filtre} filters an initial block $Z_0$ (e.g., a set of Ritz vectors) to produce an improved subspace $Z_k$. We incorporate this into the block Lanczos procedure as described next.

\nd 	In the block Lanczos-Chebyshev algorithm, we use Ho's algorithm \cite{ho2} to determine the optimal ellipse. The complete procedure for computing the $s$ eigenvalues of largest real part of a large nonsymmetric matrix is summarized as follows:

\begin{algorithm}
	\caption{Block Lanczos--Chebyshev algorithm}\label{blac}
	\begin{itemize}
		\item \textbf{Step 1}: Choose initial $V_1 = [v_1,\ldots,v_s]\in\R^{n\times s}$ and $W_1$ such that $W_1^T V_1 = I_s$. Set the maximum Lanczos steps $m$ and the Chebyshev polynomial degree $k$.
		%Choose initial $V_1 = [v_1, \ldots, v_s] \in \R^{n \times s}$ and $W_1$ such that $W_1^T V_1 = I_s$. Set the maximum Lanczos steps $m$ and the Chebyshev polynomial degree $k$.
		\item \textbf{Step 2 (Block Lanczos stage)}: Starting with $V_1, W_1$, run Algorithm 1 for $m$ steps to obtain the block tridiagonal matrix $\widetilde{T}_m$. Compute the eigenvalues of $\widetilde{T}_m$ and select the $s$ eigenvalues of largest real part (and their Ritz vectors).
		%(Block Lanczos stage) Starting with $V_1, W_1$, run Algorithm~1 for $m$ steps to obtain the block tridiagonal matrix $\tilde{T}_m$. Compute the eigenvalues of $\tilde{T}_m$ and select $\mu_1, \ldots, \mu_s$ as the eigenvalues of largest real part. Compute the corresponding Ritz vectors of $A$ and their residual norms for convergence testing.
		\item \textbf{Step 3}: Determine the optimal Chebyshev ellipse that encloses the unwanted eigenvalues $ Spec(\widetilde{T}_m)\setminus{\mu_1,\ldots,\mu_s}$ (using, e.g., the procedure of Ho \cite{ho2}).
		% Determine the optimal Chebyshev ellipse that encloses the unwanted eigenvalues $\mathrm{Spec}(\tilde{T}_m)\setminus\{\mu_1, \ldots, \mu_s\}$ (using, e.g., the procedure of Ho \cite{ho2}).
		\item  \textbf{Step 4 (Chebyshev stage)}: Let $Z_0$ be the $n\times s$ matrix whose columns are the $s$ Ritz vectors associated with $\mu_1,\ldots,\mu_s$. Perform $k$ steps of the Chebyshev iteration (Algorithm 3) to obtain a new starting block $V_1 := Z_k$. Re-orthogonalize $Z_k$ (via QR) to obtain the new $V_1$, and choose $W_1$ appropriately to satisfy $W_1^T V_1=I_s$.
		% (Chebyshev stage) Let $Z_0$ be the $n \times s$ matrix whose columns are the Ritz vectors associated with $\mu_1, \ldots, \mu_s$. Perform $k$ steps of the Chebyshev iteration defined above to obtain a new starting block $V_1 = Z_k$. Choose a new $W_1$ such that $W_1^T V_1 = I_s$.
		\item \textbf{Step 5}: Repeat from Step 2 until convergence or until a maximum number of restarts is reached.
		%If convergence criteria are not met, go to Step 2 (restart).
	\end{itemize}
\end{algorithm}

This algorithm combines the block Lanczos method with Chebyshev polynomial filtering to efficiently compute a few of the dominant eigenpairs (eigenvalues and eigenvectors) of $A$. The Chebyshev iteration acts as a filter to improve the quality of the starting vectors before each Lanczos restart, thereby accelerating convergence to the desired eigenvalues.
%\newpage
\section{The Adaptive Block Lanczos-Chebyshev Method}\label{sec6}
\subsection{The adaptive block Lanczos algorithm}
One of the main challenges with both the standard Lanczos algorithm and its block variant is the potential for numerical instabilities, such as loss of biorthogonality or near-breakdown due to rank deficiency. To address these issues, we adopt a more stable variant introduced in \cite{bai}, known as the Adaptive Block Lanczos method (ABLE). This method enhances the robustness of the algorithm by incorporating stabilization mechanisms to mitigate the degradation of orthogonality during the iterations. The ABLE algorithm is detailed in Algorithm~\ref{blac2}. The method constructs two sequences of block vectors \( \{V_j\} \) and \( \{W_j\} \), which form biorthogonal bases satisfying:
\[
W_j^\top V_j = I_s, \qquad \mbox{~~for each iteration~~ } j.
\]
At each step, the action of \( A \) and \( A^\top \) on the current bases $A V_j$ is orthogonalized against previous $V$’s, and similarly for $A^T W_j$. This leads to a block tridiagonal projected matrix \( T_m \in \mathbb{R}^{ms \times ms} \), whose eigenvalues (Ritz values) approximate those of the original matrix \( A \).

One of the main challenges in such algorithms is the gradual loss of biorthogonality, which can significantly degrade numerical stability and the quality of the computed approximations. To mitigate this, ABLE introduces a stabilization mechanism through a singular value decomposition (SVD) of the inner product between the newly generated blocks:
\[
W_{j+1}^\top V_{j+1} = P_j D_j Q_j^\top.
\]
This is followed by a symmetric rescaling step:
\[
V_{j+1} \leftarrow V_{j+1} Q_j D_j^{-1/2}, \qquad
W_{j+1} \leftarrow W_{j+1} P_j D_j^{-1/2}.
\]
This ensures that the new blocks remain biorthogonal and averts numerical rank-deficiency that would be caused by nearly linearly dependent vectors. ABLE is described  in Algorithm \ref{blac2}.\\
%The algorithm is detailed in Algorithm~\ref{blac2}.
\begin{algorithm}
	\caption{The   Adaptive Block Lanczos algorithm (ABLE)}\label{blac2}
	\begin{enumerate}
		\item \textbf{Inputs:} $A\in \mathbb{R}^{n \times n}, V, W \in \mathbb{R}^{n \times p}$ and $m \in \mathbb{N}$.
		\item \textbf{Initialize :} $V_{0}=W_{0}=0_{p}$ and $C_{1}=B_{1}=0_{p}$.
		\item For j=1,...,m
		\item \ \ \ $S_{j}=AV_{j}$ and $R_{j}=A^{T}W_{j}$
		\item \ \ \ $A_{j}=W_{j}^{T}S_{j}, B_{j}=W_{j-1}^{T}S_{j}$ and $C_{j}^{T}=V_{j-1}^{T}R_{j}$
		\item \ \ \ $S_{j+1}=S_{j}-V_{j}A_{j}-V_{j-1}B_{j}$ and $R_{j+1}=R_{j}-W_{j}A_{j}^{T}-W_{j-1}C_{j}^{T}$
		\item \ \ \ Compute the QR decomposition $S_{j+1}=V_{j+1}C_{j+1}^{T}$ and $R_{j+1}=W_{j+1}B_{j+1}^{T}$
		\item \ \ \ Compute the SVD of  $W_{j+1}^{T}V_{j+1}=P_{j}D_{j}Q_{j}^{T}$
		\item \ \ \ Update: $V_{j+1}=V_{j+1}Q_{j}D_{j}^{-1/2}$ and $W_{j+1}=W_{j+1}P_{j}D_{j}^{-1/2}$
		\item \ \ \ Update:  $B_{j+1}=D_{j}^{1/2}Q_{j}^{T}B_{j+1}$ and $C_{j+1}=D_{j}^{1/2}P_{j}^{T}C_{j+1}$
		\item end For.
	\end{enumerate}
\end{algorithm}

\subsection{ ABLE-Chebyshev}

When the eigenvalues of interest are not well-separated from the rest of the spectrum (i.e., small spectral gap), convergence may be slow even with stabilized iterations. To improve convergence in such cases, we augment ABLE with the Chebyshev filtering technique described in previous Section.
We use the same Chebyshev filter $P_k$ defined in previous Section :
$P_k(\lambda) = \frac{T_k((\lambda - d)/c)}{T_k((\lambda_s - d)/c)},$
where $d$ and $c$ define the Chebyshev ellipse (as before) and $\lambda_s$ is a reference eigenvalue near the desired spectral region. This polynomial is constructed to satisfy $P_k(\lambda_s)=1$ and $|P_k(\lambda)|<1$ for undesired $\lambda$.. This polynomial satisfies \( P_k(\lambda_s) = 1 \) and \( P_k(\lambda) < 1 \) for unwanted \( \lambda \).

In practice, after each block Lanczos cycle of length $m$, we apply the filter to the computed Ritz vectors. If $Z_j$ is the matrix of Ritz vectors (approximate eigenvectors) obtained, we replace
$$Z_j \leftarrow P_k(A)\, Z_j$$
using the polynomial defined above.
\[
Z_j \leftarrow \frac{T_k\left( \frac{A - d I}{c} \right)}{T_k\left( \frac{\lambda_s - d}{c} \right)} Z_j.
\]
This reshapes the spectral content of the Krylov subspace, enhancing alignment with the desired eigenspace. The filtered block is then orthonormalized and used as a starting point for the next iteration, effectively realizing a \textit{thick restart}.

% Insert your structured ABLE-Chebyshev paragraph and algorithm here (already done above)

\begin{algorithm}[H]
	\caption{The Adaptive Block Lanczos--Chebyshev Algorithm (ABLE-Chebyshev)}\label{alg:ablecheb}
	\begin{enumerate}
		\item Choose initial blocks \( V_1, W_1 \in \mathbb{R}^{n \times s} \) such that \( W_1^\top V_1 = I \). Set parameters \( m \), \( k \), \( d \), and \( c \).
		\item \textbf{Block Lanczos Stage:} Run ABLE for \( m \) steps to get the projected matrix \( T_m \) and dominant Ritz values \( \mu_1, \dots, \mu_s \), with associated Ritz vectors \( Z \).
		\item \textbf{Ellipse Estimation:} Determine the Chebyshev ellipse (center \( d \), radius \( c \)) that excludes \( \mu_1, \dots, \mu_s \) but includes unwanted spectral regions.
		\item \textbf{Chebyshev Filtering:} Apply the recurrence:
		\[
		Z_1 = \frac{\sigma_1}{c} (A - d I) Z_0, \quad Z_{i+1} = 2 \frac{\sigma_{i+1}}{c}(A - d I)Z_i - \sigma_i \sigma_{i+1} Z_{i-1}
		\]
		%	\item Orthonormalize the filtered block \( Z_k \), and set \( V_1 := \texttt{qr}(Z_k), \; W_1 := V_1 \) to preserve biorthogonality.
		\item Re-orthonormalize \( Z_k \) and restart the process with \( V_1 = \mbox{qr}(Z_k), W_1 = V_1 \). Repeat until convergence.
	\end{enumerate}
\end{algorithm}
The ABLE-Chebyshev algorithm combines the numerical robustness of SVD-based block Lanczos with the spectral targeting of polynomial filtering. While the Chebyshev filter does not reduce the residual norm directly, it improves the quality of the search subspace and helps guide the Lanczos process toward the dominant spectral components (i.e., the components associated with $\lambda_1,\ldots,\lambda_s$). This makes ABLE-Chebyshev particularly effective for non-Hermitian eigenvalue problems with clustered or tightly packed spectra.

\section{Numerical examples}\label{sec7}
For the numerical validation, we present several examples designed to compare the performance of the proposed block Lanczos variants, specifically, the standard Adaptive Block Lanczos (ABLE) method and its Chebyshev-filtering counterpart (ABLE–Chebyshev). These experiments assess both the convergence rate and the accuracy of the computed dominant eigenvalues. The test problems include both dense and sparse nonsymmetric matrices, as well as a symmetric benchmark arising from the discretized 2D Laplacian.
\\
All experiments were implemented in MATLAB R2016a and executed on a machine equipped with an Intel Core i3 processor (2.4 GHz) and 4 GB of RAM. Each computation was performed using double precision arithmetic, yielding approximately 15–16 significant decimal digits of numerical accuracy.
\\
We examine the impact of Chebyshev filtering on convergence error and residual norms, and highlight cases where the standard ABLE method struggles due to spectral clustering or small gaps, while the ABLE–Chebyshev variant succeeds in efficiently isolating the desired eigenvalues.

%	\textcolor{red}{\bf Faire plus de tests numériques}

\subsection*{Example 1}

To evaluate the impact of Chebyshev filtering on block Krylov methods, we consider a dense, non-symmetric, diagonalizable matrix \( A \in \mathbb{R}^{n \times n} \) with \( n = 2000 \), constructed as
\[
A = Q D Q^{-1},
\]
where \( D \) is a diagonal matrix with eigenvalues evenly spaced from 10 down to 1, and random invertible matrix  \( Q \) (full-rank). This ensures that \( A \) is diagonalizable with a known spectrum, but not symmetric. We compute the \( s \) dominant eigenvalues of \( A \) using two algorithms. The first is ABLE, the adaptive block Lanczos algorithm without Chebyshev filtering, which serves as the baseline. The second is ABLE--Chebyshev, the proposed adaptive block Lanczos algorithm combined with Chebyshev filtering, using a polynomial filter of degree \( k \) and ellipse parameters \( d \) and \( c \) as defined earlier.  Both the ABLE and ABLE--Chebyshev methods were executed under identical computational settings, summarized in Table~\ref{tab:params1}.:

\begin{table}[H]
	\centering
		\caption{Summary of computational parameters used in Example 1}
	\label{tab:params1}
	\begin{tabular}{l|l}
		\hline
		Parameter & Value \\
		\hline
		Matrix size $n$ & 2000 \\
		Number of eigenvalues $s$ & 5 \\
		Krylov subspace dimension $m$ & 30 \\
		Chebyshev polynomial degree $k$ & 20 \\
		Maximum number of restarts & 10 \\
		Convergence tolerance & $10^{-8}$ \\
		\hline
	\end{tabular}
\end{table}
Figure \ref{fig:able-convergence1} illustrates the evolution of the residual norm over successive iterations for both the ABLE and ABLE–Chebyshev algorithms. As seen in the plot, the residual norm for the standard ABLE method remains nearly constant across iterations, indicating stagnation and failure to make meaningful progress toward convergence. In contrast, the ABLE–Chebyshev variant exhibits a consistent and substantial decrease in the residual norm, with a drop of several orders of magnitude within the first few iterations. This demonstrates the effectiveness of Chebyshev filtering in enhancing the quality of the Krylov subspace and guiding the approximation toward the dominant eigenspace. The results confirm that the ABLE–Chebyshev method is significantly more robust in resolving the targeted eigenvalues, especially in challenging spectral configurations.

\begin{figure}[h!]
	\centering
	\includegraphics[width=0.6\textwidth]{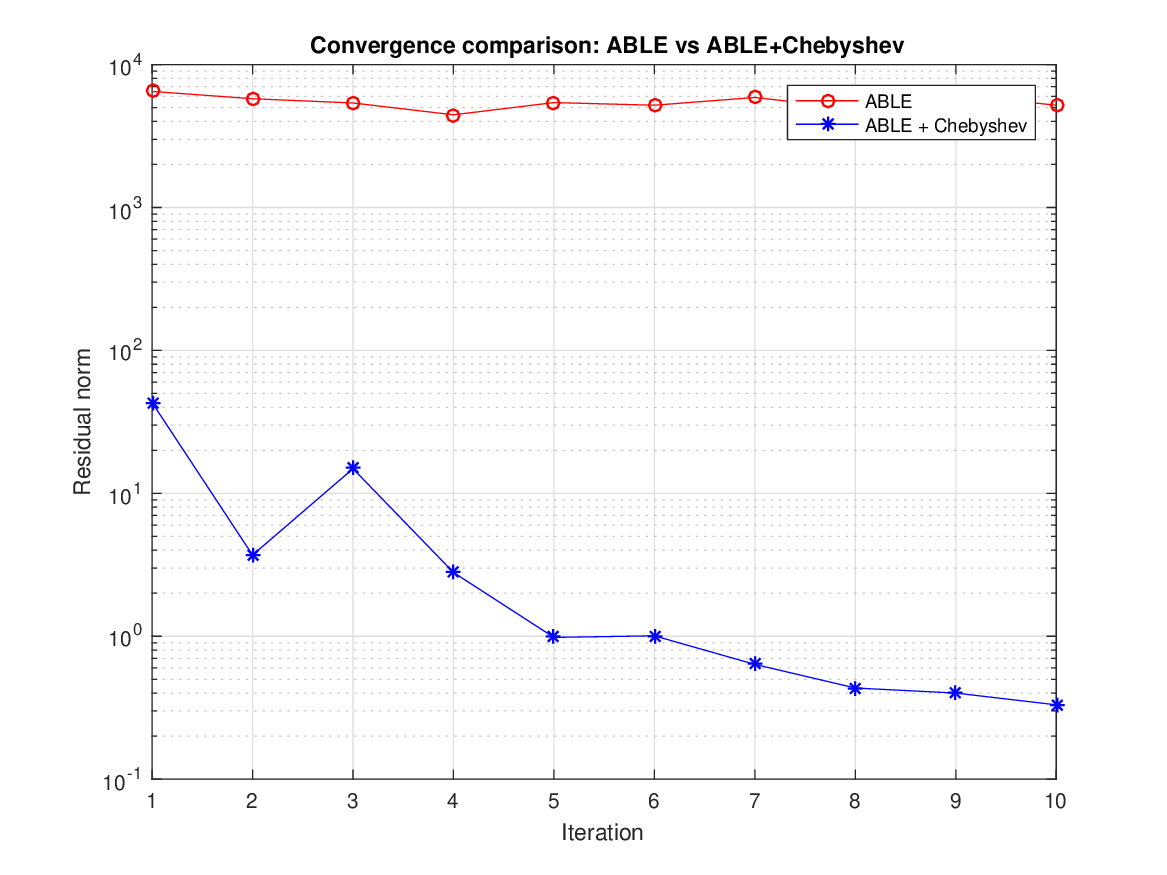}
	\caption{Residual norm per iteration for ABLE, ABLE-- Chebyshev}
	\label{fig:able-convergence1}
\end{figure}
Table~\ref{tab:able-errors1} reports the absolute error between the computed Ritz values and the exact eigenvalues extracted from the diagonal of \( D \). The results clearly highlight the superiority of the ABLE--Chebyshev method over the standard ABLE approach. While the baseline ABLE method yields large errors—reaching up to \( \mathcal{O}(10^2) \) for some indices—the Chebyshev-accelerated variant consistently achieves significantly higher accuracy across all five dominant eigenvalues. Notably, for the fifth eigenvalue, the error is reduced from \( 1.10 \times 10^{2} \) to \( 5.00 \times 10^{-4} \), demonstrating the effectiveness of spectral filtering in isolating the desired spectral components. This confirms that Chebyshev filtering not only enhances convergence behavior but also improves the reliability of the eigenvalue approximations.

\begin{table}[H]
	\centering
	\caption{Absolute error on the largest 5 eigenvalues}
	\label{tab:able-errors1}
	\begin{tabular}{c|c|c|c}
		\toprule
		\textbf{Index} & \textbf{ABLE}  & \textbf{ABLE-- Chebyshev} \\
		\midrule
		1 & $ 2.48 \times 10^{+2}$ & $ 1.29 \times 10^{-2}$  \\
		2 & $ 7.98 \times 10^{0}$ & $ 9.55 \times 10^{-3}$  \\
		3 & $ 1.45 \times 10^{+1}$ & $ 5.06 \times 10^{-3}$  \\
		4 & $ 8.35 \times 10^{+1}$ & $ 1.84 \times 10^{-3}$  \\
		5 & $ 1.10 \times 10^{+2}$ & $ 5.00 \times 10^{-4}$  \\
		\bottomrule
	\end{tabular}
\end{table}

\subsection*{Example 2}
To further evaluate the robustness of the proposed methods, we repeat the numerical experiment using a larger matrix and modified algorithmic parameters. The updated configuration is summarized in Table~\ref{tab:params2}.
\begin{table}[H]
	\caption{Computational parameters used in Example 2}
	\label{tab:params2}
	\centering
	\begin{tabular}{l|l}
		\hline
		Parameter & Value \\
		\hline
		Matrix size $n$ & 5000 \\
		Number of eigenvalues $s$ & 5 \\
		Krylov subspace dimension $m$ & 60 \\
		Chebyshev polynomial degree $k$ & 60 \\
		Maximum number of restarts & 30 \\
		Convergence tolerance & $10^{-8}$ \\
		\hline
	\end{tabular}
\end{table}
Figure~\ref{fig:able-convergence2} and Table~\ref{tab:able-errors2} confirm the findings observed earlier: the ABLE--Chebyshev algorithm consistently outperforms the baseline ABLE method in terms of convergence behavior and approximation accuracy. 
\\
In Figure~\ref{fig:able-convergence2}, the residual norm for the ABLE method stagnates around \(10^4\) over the course of 30 iterations, indicating a failure to progress. In contrast, the ABLE--Chebyshev method demonstrates a steady and substantial reduction in the residual norm, reaching values close to \(10^0\), which indicates that the filtered Krylov subspace better captures the dominant eigenspace of the matrix.
\\
Similarly, Table~\ref{tab:able-errors2} shows that the absolute errors in the computed eigenvalues using ABLE remain large—up to \(2.08 \times 10^3\) for the first eigenvalue—while the ABLE--Chebyshev variant reduces all errors to within at most \(1.31 \times 10^{-1}\), confirming the enhanced reliability of the Chebyshev-filtered strategy across different problem settings.
\begin{table}[H]
	\centering
	\caption{Absolute error on the dominant 5 eigenvalues}
	\label{tab:able-errors2}
	\begin{tabular}{c|c|c|c}
		\toprule
		\textbf{Index} & \textbf{ABLE}  & \textbf{ABLE-- Chebyshev} \\
		\midrule
		1 & $ 2.08 \times 10^{+3}$ & $ 6.37 \times 10^{-3}$  \\
		2 & $ 1.67 \times 10^{+2}$ & $ 2.00 \times 10^{-2}$  \\
		3 & $ 3.03 \times 10^{+1}$ & $ 3.37 \times 10^{-2}$  \\
		4 & $ 1.08 \times 10^{+2}$ & $ 4.17 \times 10^{-2}$  \\
		5 & $ 2.27 \times 10^{+2}$ & $ 1.31 \times 10^{-1}$  \\
		\bottomrule
	\end{tabular}
\end{table}

\begin{figure}[H]
	\centering
	\includegraphics[width=0.6\textwidth]{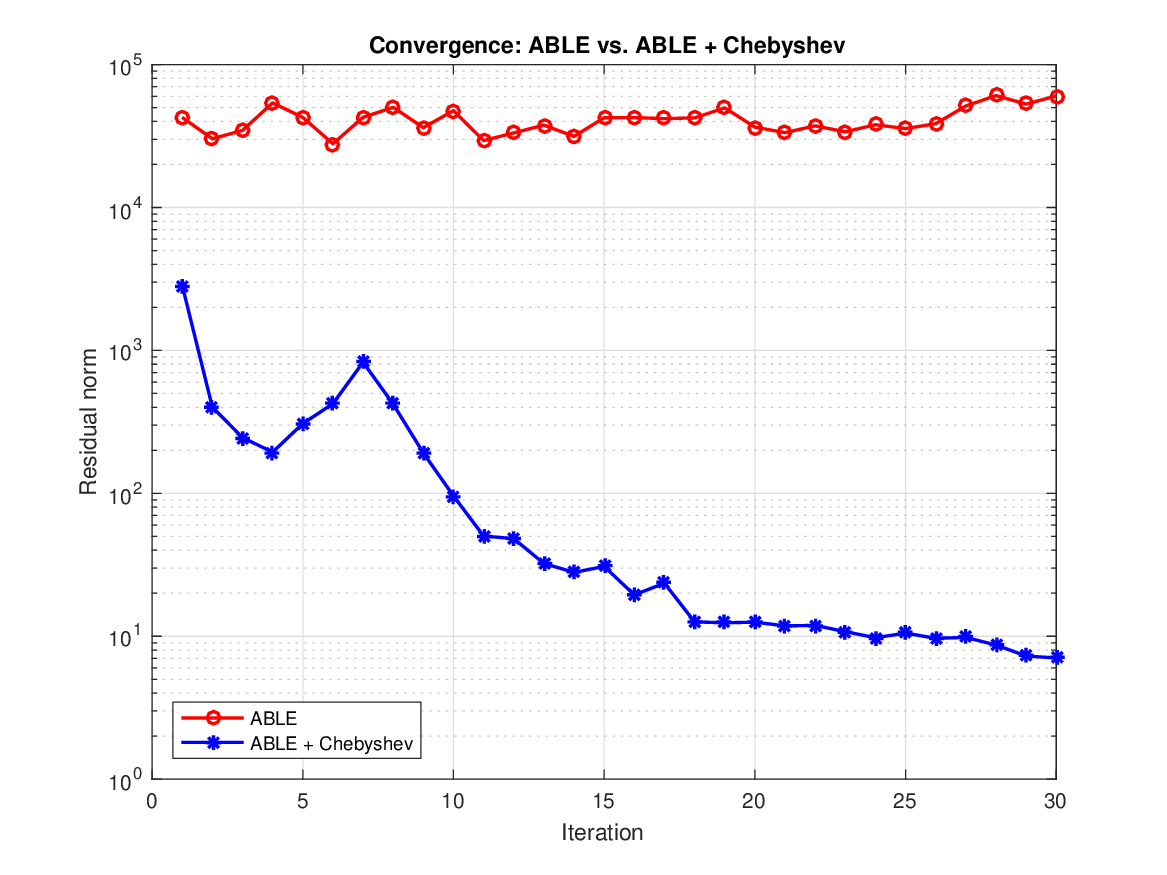}
	\caption{Residual norm per iteration for ABLE, ABLE-Chebyshev}
	\label{fig:able-convergence2}
\end{figure}

%%%%%%%%%%%%%%%%%%%%%%%%%%%%%%%%%%%%%%%%%%%%%%%%%%%%%%%%%%
\subsection*{Example 3: Symmetric Eigenvalue Problem from the 2D Laplacian}

We consider the symmetric eigenvalue problem associated with the negative Laplacian operator defined on the unit square with homogeneous Dirichlet boundary conditions:
\begin{equation}
	\begin{cases}
		- \Delta u(x, y) = \lambda u(x, y), & (x, y) \in \Omega = (0,1)^2, \\
		u(x, y) = 0, & (x, y) \in \partial \Omega.
	\end{cases}
\end{equation}
This problem arises in various physical models, such as steady-state heat conduction, vibration of membranes, or diffusion processes. 
We discretize the domain $\Omega$ using a uniform grid of $N \times N$ interior points, with mesh spacing $h = \frac{1}{N+1}$. The Laplace operator is approximated using the standard second-order finite difference scheme:
\begin{equation}
	- \Delta u(x_i, y_j) \approx \frac{1}{h^2} \left( -u_{i-1,j} - u_{i+1,j} - u_{i,j-1} - u_{i,j+1} + 4u_{i,j} \right).
\end{equation}
This leads to a discrete matrix eigenvalue problem of the form:
\begin{equation}
	A \mathbf{u} = \lambda \mathbf{u},
\end{equation}
where $A \in \mathbb{R}^{N^2 \times N^2}$ is a sparse, real, symmetric, and positive definite matrix with a block tridiagonal structure.
After applying a 5-point finite difference discretization of the Laplace operator on a square grid of size \( N \times N \), we obtain a sparse symmetric matrix \( A \in \mathbb{R}^{n \times n} \), where \( n = N^2 \). The matrix \( A \) has the following block structure:
\[
A = \begin{bmatrix}
	T & -I &      &        &        \\
	-I & T & -I   &        &        \\
	& -I & T   & \ddots &        \\
	&    & \ddots & \ddots & -I   \\
	&    &        & -I & T
\end{bmatrix} \in \mathbb{R}^{N^2 \times N^2},
\]
where \( T \in \mathbb{R}^{N \times N} \) is the tridiagonal matrix
\[
T = \begin{bmatrix}
	4 & -1 &        &        &        \\
	-1 & 4 & -1     &        &        \\
	& \ddots & \ddots & \ddots &        \\
	&        & -1 & 4 & -1     \\
	&        &     & -1 & 4
\end{bmatrix}, \quad
I = \text{Identity matrix of size } N.
\]
Each block row in \( A \) corresponds to one row of the 2D grid, and each entry connects a point to its horizontal and vertical neighbors, reflecting the 5-point stencil:
\[
\begin{array}{ccc}
	& -1 & \\
	-1 & 4  & -1 \\
	& -1 &
\end{array}
\]
The matrix \( A \) is real, symmetric, and positive definite. It has at most 5 nonzero entries per row, and its sparsity pattern corresponds to the connectivity of interior nodes in a regular 2D grid. 
The matrix $A$ can be expressed using Kronecker products as:
\begin{equation}
	A = \frac{1}{h^2} \left( I_N \otimes T_N + T_N \otimes I_N \right),
\end{equation}
where $T_N = \text{tridiag}(-1, 4, -1) \in \mathbb{R}^{N \times N}$ is the 2D discrete Laplacian and $I_N$ is the $N \times N$ identity matrix.\\
In our experiment, we set $N = 50$, which results in a matrix of size $N^2 = 2500$. The matrix is assembled using Kronecker products and is guaranteed to be symmetric by construction.
\\
Table~\ref{tab:able-errors-lap1} presents the approximate eigenvalues computed by the ABLE--Chebyshev method at convergence, alongside the corresponding results from the standard ABLE algorithm. The exact eigenvalues are known in this case, as the problem is derived from the symmetric discretization of the 2D Laplacian. For the first two eigenvalues, both methods achieve full precision with errors on the order of \(10^{-15}\), consistent with the use of double-precision arithmetic. However, for the third and fourth eigenvalues, the baseline ABLE method fails to resolve the correct spectral values, returning approximations that are orders of magnitude away from the exact ones. In contrast, the ABLE--Chebyshev method successfully captures these eigenvalues with errors on the order of \(10^{-2}\) and \(10^{-3}\), respectively. These results confirm that Chebyshev filtering significantly enhances the robustness and reliability of the eigenvalue approximations, especially when the dominant part of the spectrum is not well-isolated.

\begin{table}[H]
	\centering
	\caption{Comparison of exact and computed eigenvalues for a symmetric Laplacian problem using ABLE and ABLE--Chebyshev methods}
	\label{tab:able-errors-lap1}
	\begin{tabular}{|c|c|c|c|c|c|}
		\hline
		\textbf{Index} & \textbf{Exact} & \textbf{ABLE} & \textbf{ABLE-Cheb} & \textbf{ Error (ABLE)} & \textbf{Error (ABLE-Cheb)} \\
		\hline
		1 & \(7.9924\) & \(7.9924\) & \(7.9924\) & \(8.88 \times 10^{-16}\) & \(1.78 \times 10^{-15}\) \\
		2 & \(7.9810\) & \(7.9810\) & \(7.9810\) & \(1.78 \times 10^{-15}\) & \(7.11 \times 10^{-15}\) \\
		3 & \(7.9810\) & \(1.8952 \times 10^{-2}\) & \(7.9697\) & \(7.9620\) & \(1.14 \times 10^{-2}\) \\
		4 & \(7.9697\) & \(7.5867 \times 10^{-3}\) & \(7.9622\) & \(7.9621\) & \(7.53 \times 10^{-3}\) \\
		\hline
	\end{tabular}
\end{table}
Table~\ref{tab:able-errors-lap2} reports the residual norms associated with the eigenvalue approximations computed by the ABLE and ABLE--Chebyshev methods. The results show that both methods achieve residuals on the order of \(10^{-9}\) to \(10^{-10}\), confirming that the approximated eigenpairs satisfy the eigenvalue relation to high numerical accuracy. Notably, the residual norms are comparable between the two methods, which indicates that the improved eigenvalue approximations achieved by ABLE--Chebyshev, as shown previously in Table~\ref{tab:able-errors-lap1}, are not the result of overfitting or loss of numerical consistency. This further validates the reliability of the Chebyshev-filtering approach in resolving dominant eigenmodes while maintaining numerical stability.

\begin{table}[h!]
	\centering
	\caption{Residual norms associated with ABLE and ABLE--Chebyshev methods.}
	\label{tab:able-errors-lap2}
	\begin{tabular}{|c|c|c|}
		\hline
		\textbf{Index} & \textbf{Residual (ABLE)} & \textbf{Residual (ABLE-Cheb)} \\
		\hline
		1 & \(1.09 \times 10^{-10}\) & \(2.20 \times 10^{-10}\) \\
		2 & \(1.98 \times 10^{-09}\) & \(9.24 \times 10^{-09}\) \\
		3 & \(1.08 \times 10^{-08}\) & \(1.03 \times 10^{-08}\) \\
		4 & \(1.32 \times 10^{-09}\) & \(1.12 \times 10^{-09}\) \\
		\hline
	\end{tabular}
\end{table}
\section{Conclusion}\label{sec8}
In this paper, we presented and studied enhanced variants of the block Lanczos method for computing dominant eigenvalues of large, nonsymmetric, diagonalizable matrices. Our approach integrates Chebyshev polynomial filtering with the ABLE method to improve the reliability and robustness of the computed approximations. The resulting ABLE--Chebyshev algorithm combines stabilized biorthogonalization with targeted spectral filtering, allowing it to handle difficult spectral configurations, such as closely clustered eigenvalues or small spectral gaps.
\\
We further showed that the convergence behavior on any diagonalizable matrix can be interpreted through its behavior on an associated normal matrix, offering useful insight into the structure of Krylov subspace methods in the nonsymmetric setting.
\\
Numerical experiments demonstrated that the ABLE--Chebyshev method consistently produces more accurate eigenvalue approximations than the baseline ABLE algorithm, while maintaining comparably low residual norms. These results were confirmed across a range of matrix types, including dense nonsymmetric examples and structured problems derived from partial differential equations.

%
% ---- Bibliography ----
%


\begin{thebibliography}{6}
%

\bibitem{bai}
\sc Z.~Bai, D.~Day and Q.~Ye, \rm ABLE: An adaptive block Lanczos method for non-Hermitian eigenvalue problems, SIAM Journal on Matrix Analysis, 20(4), 839--1098 (1999).

\bibitem{barkouki}
\sc H.~Barkouki, A.~H.~Bentbib and K.~Jbilou, \rm An adaptive rational block Lanczos-type algorithm for model reduction of large scale dynamical systems, Journal of Scientific Computing, 67(1), 221--236 (2015).

\bibitem{barkouki1}
\sc H.~Barkouki, A.~H.~Bentbib and K.~Jbilou, \rm A matrix rational Lanczos method for model reduction in large-scale first and second-order dynamical systems, Numerical Linear Algebra with Applications, 24(1), e2077 (2017).

\bibitem{bentbib}
\sc A.~H.~Bentbib, M.~El Guide and K.~Jbilou, \rm The block Lanczos algorithm for linear ill-posed problems, Calcolo, 54(3), 711--732 (2017).

\bibitem{Kamal_bentbib}
\sc A.~H.~Bentbib, A.~Kanber and K.~Lachhab, \rm Generalized subspace iteration method for solving matrix pair eigenproblem, JMM, 7(3), 337--355 (2019).

\bibitem{bouh}
\sc A.~Bouhamidi, K.~Jbilou and M.~Raydan, \rm Convex constrained optimization for large-scale generalized Sylvester equations, Computational Optimization and Applications, 48, 233--253 (2011).

\bibitem{cullum}
\sc J.~Cullum, \rm Arnoldi versus nonsymmetric Lanczos algorithms for solving matrix eigenvalue problems, BIT, 36(3), 470--493 (1996).

\bibitem{elguennouni}
\sc A.~El~Guennouni, K.~Jbilou and H.~Sadok, \rm The block Lanczos method for linear systems with multiple right-hand sides, Applied Numerical Mathematics, 51, 243--256 (2004).

\bibitem{freund}
\sc W.~Freund, \rm Band Lanczos method, in \emph{Templates for the Solution of Algebraic Eigenvalue Problems: A Practical Guide}, Z.~Bai, J.~Demmel, J.~Dongarra, A.~Ruhe, and H.~van der Vorst (eds.), SIAM, Philadelphia, 2000, pp. 80--88.

\bibitem{golub}
\sc G.~H.~Golub and C.~F.~Van~Loan, \rm Matrix Computations, 2nd ed., The Johns Hopkins University Press, New York, 1989.

\bibitem{ho2}
\sc D.~Ho, \rm Chebyshev acceleration technique for large scale nonsymmetric matrices, Numerische Mathematik, 56, 721--734 (1990).

\bibitem{jbilou}
\sc K.~Jbilou, H.~Sadok and A.~Tinzeft, \rm Oblique projection methods for multiple linear systems, Electronic Transactions on Numerical Analysis, 20, 119--138 (2005).

\bibitem{jbilou1}
\sc K.~Jbilou, A.~Messaoudi and H.~Sadok, \rm Global FOM and GMRES algorithms for multiple right-hand sides, Applied Numerical Mathematics, 31, 49--63 (1999).

\bibitem{kerner}
\sc W.~Kerner, \rm Large-scale complex eigenvalue problems, Journal of Computational Physics, 85, 1--85 (1989).

\bibitem{lanczos}
\sc C.~Lanczos, \rm An iteration method for the solution of the eigenvalue problem of linear differential and integral operators, J. Res. Natl. Bur. Stand., 45, 225--280 (1950).

\bibitem{arpack}
\sc R.~B.~Lehoucq, D.~C.~Sorensen and C.~Yang, \rm ARPACK Users' Guide: Solution of Large Scale Eigenvalue Problems with Implicitly Restarted Arnoldi Methods, SIAM, Philadelphia, 1998.

\bibitem{mant1}
\sc T.~A.~Manteuffel, \rm The Chebyshev iteration for nonsymmetric linear systems, Numerische Mathematik, 28, 307--327 (1977).

\bibitem{mant2}
\sc T.~A.~Manteuffel, \rm Adaptive procedure for estimation of parameters for the nonsymmetric Chebyshev iteration, Numerische Mathematik, 31, 187--208 (1978).

\bibitem{Riquet}
\sc A.~J.~Riquet, \rm Méthodes de Krylov par blocs pour les équations matricielles en théorie du contrôle, Thesis, Université Littoral Côte d’Opale, Calais, France, 2002.

\bibitem{saad}
\sc Y.~Saad, \rm Iterative Methods for Sparse Linear Systems, PWS Publishing Company, 1996.

\bibitem{saad2}
\sc Y.~Saad, \rm Chebyshev acceleration techniques for solving nonsymmetric eigenvalue problems, Mathematics of Computation, 42(166), 567--588 (1984).

\bibitem{saad3}
\sc Y.~Saad, \rm Projection methods for solving large sparse eigenvalue problems, in \emph{Matrix Pencils}, Lecture Notes in Mathematics, vol. 973, Springer, Berlin, 1983, pp. 121--144.

\bibitem{saad_restart}
\sc Y.~Saad, \rm Numerical Methods for Large Eigenvalue Problems, SIAM, 2nd edition, 2011.

\bibitem{sadkane}
\sc M.~Sadkane, \rm Block Arnoldi and Davidson methods for unsymmetric large eigenvalue problems, Numerische Mathematik, 64, 195--211 (1993).

\bibitem{sadkane2}
\sc M.~Sadkane, \rm A block Arnoldi-Chebyshev method for computing the leading eigenpairs of large sparse unsymmetric matrices, Numerische Mathematik, 64, 181--193 (1993).

\bibitem{sorensen_polyfilter}
\sc D.~C.~Sorensen, \rm Implicit Application of Polynomial Filters in a \(k\)-Step Arnoldi Method, SIAM Journal on Matrix Analysis and Applications, 13(1), 357--385 (1992).

\bibitem{stewart_krylovschur}
\sc G.~W.~Stewart, \rm A Krylov–Schur Algorithm for Large Eigenproblems, SIAM Journal on Matrix Analysis and Applications, 23(3), 601--614 (2001).


\end{thebibliography}
\end{document}